\documentclass[11pt]{amsart}
\usepackage{amscd}
\usepackage{amscd}
\usepackage{latexsym}
\usepackage{pstcol,pst-plot,pst-3d}
\usepackage{mathptmx}
\usepackage{multicol}

\psset{unit=0.7cm,linewidth=0.8pt,arrowsize=2.5pt 4}
\newpsstyle{fatline}{linewidth=1.5pt}
\newpsstyle{fyp}{fillstyle=solid,fillcolor=verylight}
\definecolor{verylight}{gray}{0.97}
\definecolor{light}{gray}{0.9}
\definecolor{medium}{gray}{0.85}

\def\NZQ{\Bbb}               % the font for N,Z,Q,R,C
\def\NN{{\NZQ N}}

\def\ZZ{{\NZQ Z}}

\def\PP{{\NZQ P}}

\def\frk{\frak}               % font for "Fraktur"

\def\mm{{\frk m}}

\def\Phi{{\frk N}}
\def\opn#1#2{\def#1{\operatorname{#2}}} % to make operators
\opn\chara{char} \opn\length{\ell} \opn\pd{pd} \opn\rk{rk}
\opn\projdim{proj\,dim} \opn\injdim{inj\,dim} \opn\rank{rank}
\opn\depth{depth} \opn\grade{grade} \opn\height{height}
\opn\embdim{emb\,dim} \opn\codim{codim}

\opn\Tr{Tr} \opn\bigrank{big\,rank}
\opn\superheight{superheight}\opn\lcm{lcm}
\opn\trdeg{tr\,deg}%\emph{
\opn\reg{reg} \opn\lreg{lreg} \opn\ini{in} \opn\lpd{lpd}
\opn\size{size}\opn{\mult}{mult}
\opn\div{div} \opn\Div{Div} \opn\cl{cl} \opn\Cl{Cl}
\opn\Spec{Spec} \opn\Supp{Supp} \opn\supp{supp} \opn\Sing{Sing}
\opn\Ass{Ass} \opn\Min{Min}
\opn\Ann{Ann} \opn\Rad{Rad} \opn\Soc{Soc}
\opn\Syz{Syz} \opn\Im{Im} \opn\Ker{Ker} \opn\Coker{Coker}
\opn\Am{Am} \opn\Hom{Hom} \opn\Tor{Tor} \opn\Ext{Ext}
\opn\End{End} \opn\Aut{Aut} \opn\id{id}

\opn\nat{nat}
\opn\pff{pf}%   \pf exists already
\opn\Pf{Pf} \opn\GL{GL} \opn\SL{SL} \opn\mod{mod} \opn\ord{ord}
\opn\Gin{Gin}
\opn\Hilb{Hilb}\opn\adeg{adeg}\opn\std{std}\opn\ip{infpt}
\opn\Pol{Pol}\opn\sdepth{sdepth}\opn\sqdepth{sqdepth}\opn{\Mon}{Mon}\opn{\fdepth}{fdepth}
\opn\aff{aff} \opn\con{conv} \opn\relint{relint} \opn\st{st}
\opn\lk{lk} \opn\cn{cn} \opn\core{core} \opn\vol{vol}
\opn\link{link} \opn\star{star}
\opn\gr{gr}

\def\pot#1#2{#1[\kern-0.28ex[#2]\kern-0.28ex]}

\opn\dirlim{\underrightarrow{\lim}}
\opn\inivlim{\underleftarrow{\lim}}
\let\union=\cup
\let\sect=\cap
\let\dirsum=\oplus

\let\iso=\cong
\let\Union=\bigcup

\let\Dirsum=\bigoplus

\let\to=\rightarrow

\def\Implies{\ifmmode\Longrightarrow \else
        \unskip${}\Longrightarrow{}$\ignorespaces\fi}
\def\implies{\ifmmode\Rightarrow \else
        \unskip${}\Rightarrow{}$\ignorespaces\fi}
\def\iff{\ifmmode\Longleftrightarrow \else
        \unskip${}\Longleftrightarrow{}$\ignorespaces\fi}

\let\:=\colon
\newtheorem{Theorem}{Theorem}[section]
\newtheorem{Lemma}[Theorem]{Lemma}
\newtheorem{Corollary}[Theorem]{Corollary}
\newtheorem{Proposition}[Theorem]{Proposition}

\newtheorem{Examples}[Theorem]{Examples}

\newtheorem{Conjecture}[Theorem]{Conjecture}

\let\epsilon\varepsilon
\let\phi=\varphi
\let\kappa=\varkappa
\def\qed{\ifhmode\textqed\fi
      \ifmmode\ifinner\quad\qedsymbol\else\dispqed\fi\fi}
\def\textqed{\unskip\nobreak\penalty50
       \hskip2em\hbox{}\nobreak\hfil\qedsymbol
       \parfillskip=0pt \finalhyphendemerits=0}
\def\dispqed{\rlap{\qquad\qedsymbol}}

\opn\dis{dis}
\def\pnt{{\raise0.5mm\hbox{\large\bf.}}}

\opn\Lex{Lex}

\begin{document}

\title{How to compute the Stanley depth of a monomial  ideal}

\author{J\"urgen Herzog, Marius Vladoiu and Xinxian Zheng}

\address{J\"urgen Herzog, Fachbereich Mathematik und
Informatik, Universit\"at Duisburg-Essen, Campus Essen, 45117
Essen, Germany} \email{juergen.herzog@uni-due.de}

\address{Marius Vladoiu, Facultatea de Matematica si Informatica, Universitatea Bucuresti, Str.
 Academiei 14, Bucharest, RO-010014, Romania}  \email{vladoiu@gta.math.unibuc.ro}

\address{Xinxian Zheng, Fachbereich Mathematik und
Informatik, Universit\"at Duisburg-Essen, Campus Essen, 45117
Essen, Germany} \email{xinxian.zheng@uni-due.de}

\subjclass{13C13, 13C14, 05E99, 16W70}
\thanks{The second author was partially supported by CNCSIS grant ID-PCE no. 51/2007. He also wants to thank the University of Duisburg-Essen for the hospitality during his stay in Essen. The third author is grateful for the financial support by DFG (Deutsche Forschungsgemeinschaft) during the preparation of this work}

\maketitle

\begin{abstract}
Let $J\subset I$ be monomial ideals. We show that the Stanley depth of $I/J$ can be computed in a finite number of steps. We also introduce the  $\fdepth$ of a monomial ideal which is defined in terms of prime filtrations and show that it can also be computed in a finite number of steps. In both cases it is shown that these invariants can be determined by considering partitions of suitable finite posets into intervals.
\end{abstract}

\section*{Introduction}
Let $K$ be a field, $S=K[x_1,\ldots, x_n]$ be the polynomial ring in
$n$ variables, and  $M$ be a finitely generated $\ZZ^n$-graded $S$-module. Let $u\in M$ be a homogeneous element in $M$ and $Z$ a subset of
$\{x_1,\ldots, x_n\}$. We denote by $uK[Z]$ the $K$-subspace of
$M$ generated by  all elements $uv$ where $v$ is a
monomial in $K[Z]$. The $\ZZ^n$-graded $K$-subspace  $uK[Z]\subset M$ is called a
{\em Stanley space of dimension $|Z|$}, if $uK[Z]$ is a free $K[Z]$-module.

A {\em Stanley decomposition} of $M$ is a presentation of the $\ZZ^n$-graded $K$-vector space $M$ as a finite direct sum of Stanley spaces
\[
{\mathcal D}: M=\Dirsum _{i=1}^m u_iK[Z_i]
\]
in the category of $\ZZ^n$-graded $K$-vector spaces. In other words, each of the summands is a $\ZZ^n$-graded $K$-subspace of $M$ and the decomposition is compatible with the $\ZZ^n$-grading, i.e.\ for each $a\in\ZZ^n$ we have $M_a=\Dirsum _{i=1}^m (u_iK[Z_i])_a$. The number  $\sdepth {\mathcal D}=\min\{|Z_i|\:\; i=1,\ldots,m\}$ is called the {\em Stanley depth of $\mathcal D$}. The  {\em Stanley depth} of $M$ is defined to be
\[
\sdepth M=\max\{\sdepth {\mathcal D}\: \; \text{$\mathcal D$ is a Stanley decomposition of $M$}\}.
\]
It is conjectured by Stanley \cite{St} that $\depth M\leq \sdepth M$ for all $\ZZ^n$-graded $S$-modules $M$. The conjecture is widely open (see however \cite{Ap}, \cite{HeJaYa}, \cite{HePo} and \cite{Ja}). A priori it is not clear how one can compute $\sdepth M$. We will discuss this question in a special case.

Let $J\subset I\subset S$ be two monomial ideals. Then $I/J$ is a $\ZZ^n$-graded $S$-module. One of the  aims of this paper is to show that  $\sdepth I/J$ can be computed in a finite number of steps. To this end we fix an integer vector $g\in \ZZ^n$ with the property $a\leq g$ for all $a\in \ZZ^n$ with $x^a\in I\setminus J$. Here $\leq$ denotes  the partial order in $\ZZ^n$ which is given by componentwise comparison, and for $a=(a(1),\ldots, a(n))$ we denote as usual by $x^a$ the monomial $ x_1^{a(1)}\cdots x_n^{a(n)}$. Given these data, we define the {\em characteristic poset} $P^g_{I/J}$  of $I/J$ with respect to $g$ as the subposet
$$P^g_{I/J}=\{a\in \ZZ^n\:\; x^a\in I\setminus J,\quad a\leq g\}$$
of $\ZZ^n.$

As one of the main results of this paper we show in Theorem~\ref{partition} that each partition of $P^g_{I/J}$ into intervals induces a Stanley decomposition of $I/J$, and show in Theorem~\ref{sdepth}  that  for any Stanley decomposition of $I/J$ there exists one induced by a partition of $P^g_{I/J}$ whose Stanley depth is greater than or equal to the given one. These two facts together imply that the Stanley depth can be computed by considering the finitely many different partitions of $P^g_{I/J}$.

Being able to compute the Stanley depth in a finite number of steps does however not mean that we have an algorithm to compute the Stanley depth. The known algorithms (see \cite{J}, \cite{Ap} and \cite{PlRo}) to compute at least one Stanley decomposition, among them the Janet algorithm, practically never provides  a Stanley decomposition whose Stanley depth coincides with the Stanley depth of the module. For example, if we take the graded maximal ideal $\mm=(x_1,\ldots, x_n)$. Then the Janet algorithm gives a  decomposition of Stanley depth $1$. On the other hand, by using our methods we can show that $\sdepth \mm=\lceil n/2 \rceil$ for $n\leq 9$. Probably this is true for all $n$, but we do not know the general  result.  To prove this one would have to find appropriate partitions of $P_{\mm}$. To   find   the general strategy to get such partitions  in this particular case is an interesting combinatorial problem which we could not yet solve.

There is a natural lower bound for both,  $\depth M$ and $\sdepth M$. In order to  describe this bound, let
\[
{\mathcal F}\: 0=M_0\subset M_1\subset \cdots \subset M_m=M
\]
be a chain of $\ZZ^n$-graded submodules of $M$. Then $\mathcal F$  is called a {\em prime filtration} of $M$ if $M_i/M_{i-1}\iso (S/P_i)(-a_i)$ where $a_i\in \ZZ^n$ and where each $P_i$ is a monomial prime ideal. We call the  set of prime ideals $\{P_1,\ldots, P_m\}$ the {\em support} of $\mathcal{F}$ and denote it by $\supp\mathcal{F}$. Furthermore we set $\fdepth \mathcal{F}=\min\{\dim S/P\:\; P\in\supp \mathcal{F}\}$ and
\[
\fdepth M=\max\{\fdepth \mathcal{F}\:\; \text{$\mathcal{F}$ is a prime filtration of $M$}\}.
\]
It is then very easy to see that $\fdepth M\leq \depth M, \sdepth M$. Again it is not at all obvious how to actually compute the $\fdepth$ of a module. Similarly as for the $\sdepth$ we show however that the $\fdepth$ of $I/J$ can be computed in a finite number of steps. This result is a consequence of Theorem~\ref{primepartition} and Theorem~\ref{inducedprime}. Indeed, these theorems imply that $\fdepth I/J$ can be computed by considering among the  partitions of $P_{I/J}^g$ into intervals precisely those partitions which satisfy the condition that their partial unions in a suitable order are poset ideals of $P_{I/J}^g$, see Corollary~\ref{ffinite} for details.

In the last section  of this paper we present a few applications of the general theory developed in Section~2 and give some classes of examples. In  particular we prove in Proposition~\ref{complete} that any ideal monomial complete intersection satisfies Stanley's conjecture, and in Proposition~\ref{borel}  that  any ideal of Borel type satisfies Stanley's conjecture. In the case of a complete intersection we actually show that the $\fdepth$ coincides with the $\depth$. The proof of  Proposition~\ref{borel} is based on two results shown before in this section. The first result (Proposition~\ref{lower}) says that the  $\sdepth$ of a monomial ideal is bounded below by $n-m+1$ where $n$ is the number of variables of the ambient polynomial ring and where $m$ is the number of generators of the ideal.  The second result needed in the proof of Proposition~\ref{borel} says that the $\sdepth$ of the extension of  a monomial ideal in a polynomial extension goes up by the number of variables which are adjoined in this extension, see Proposition~\ref{extension}. We also compute the Stanley depth of any complete intersection generated by three elements. It turns out that its Stanley depth is always equal to $n-1$. In a final observation we show that the conjecture of Soleyman Jahan \cite{Ja1} concerning a lower bound for the regularity of a $\ZZ^n$-graded module implies the following conjecture: there exists a partition $P^g_{I/J}=\Union_{I=1}^r[c_i,d_i]$ of $P^g_{I/J}$ with the property that $|c_i|\leq \reg I/J$ for all $i$. Here $|c|$ denotes the sum of the components of the vector $c$.

\section{Stanley decompositions and prime filtrations}

In this section we shall discuss the relationship between Stanley decompositions and prime filtrations. We  will also recall some basic upper and lower bounds for the Stanley depth.

Let $K$ be a field. Throughout the paper $S$ will denote the polynomial ring $K[x_1,,\ldots,x_n]$ in $n$ variables over $K$. Figure \ref{Fig1}  displays  a Stanley decomposition of $S/I$ and of $I$ for the monomial ideal $I=(x_1x_2^3,x_1^3x_2)$. The gray area represents the $K$-vector space spanned by the monomials in $I$. The hatched area, the fat lines and the isolated fat dots represent Stanley spaces of dimension 2,1, and $0$, respectively. According
to Figure~\ref{Fig1} we have the following Stanley decompositions
\[
I=x_1x_2^3K[x_1,x_2]\dirsum x_1^3x_2^2K[x_1]\dirsum  x_1^3x_2K[x_1],
\]
and
\[
S/I=K[x_2]\dirsum x_1K[x_1]\dirsum x_1x_2K \dirsum x_1x_2^2K \dirsum x_1^2x_2K\dirsum x_1^2x_2^2K.
\]
Here we identify $S/I$ with the $K$-subspace of $S$  generated by all monomials $u\in S\setminus I$.

\begin{figure}[hbt]
\begin{center}
\psset{unit=1.2cm}
\begin{pspicture}(0,1)(5,4)
\pspolygon*[linewidth=2pt, linecolor=lightgray](1.53,2.53)(4.5,2.53)(4.5,3.5)(1.53,3.5)
\pspolygon*[linewidth=2pt, linecolor=lightgray](2.52,2.48)(2.52,2.035)(4.5,2.035)(4.5,2.48)
\pspolygon*[linewidth=2pt, linecolor=lightgray](2.52,1.975)(2.52,1.52)(4.5,1.52)(4.5,1.975)
\psframe[fillstyle=vlines,hatchangle=45,fillcolor=black,linestyle=none](1.53,2.53)(4.5,3.5)
\psline(1,1)(4.5,1)
\psline[linewidth=1.6pt](1,1)(1,3.5)
\psline[linewidth=1.6pt](1.5,1)(4.5,1)
\psline(1.5,2.5)(1.5,3.5)
\psline(1.5,2.5)(2.5,2.5)
\psline(2.5,2.5)(2.5,1.5)
\psline[linewidth=1.6pt](2.5,1.5)(4.5,1.5)
\psline[linewidth=1.6pt](2.5,2)(4.5,2)
\psline[linewidth=1.6pt](1.5,2.5)(4.5,2.5)
\psline[linewidth=1.6pt](1.5,2.5)(1.5,3.5)
\rput(1,1){$\bullet$}
\rput(1.5,1){$\bullet$}
\rput(2.5,1.5){$\bullet$}
\rput(2.5,2){$\bullet$}
\rput(1.5,2.5){$\bullet$}
\rput(1.5,1.5){$\bullet$}
\rput(1.5,2){$\bullet$}
\rput(2,1.5){$\bullet$}
\rput(2,2){$\bullet$}

\end{pspicture}
\end{center}
\caption{}\label{Fig1}
\end{figure}

We first note

\begin{Lemma}
\label{exists}
Any finitely generated $\ZZ^n$-graded $S$-module $M$ admits a Stanley decomposition.
\end{Lemma}

The proof is based on the fact that any prime filtration of $M$ yields a Stanley decomposition. We call a chain of $\ZZ^n$-graded submodules
\[
{\mathcal F}\: 0=M_0\subset M_1\subset \cdots \subset M_m=M
\]
 a {\em prime filtration} of $M$ if $M_i/M_{i-1}\iso (S/P_i)(-a_i)$ where $a_i\in \ZZ^n$ and where each $P_i$ is a monomial prime ideal. We call the  set of prime ideals $\{P_1,\ldots, P_m\}$ the {\em support} of $\mathcal{F}$ and denote it by $\supp(\mathcal{F})$.

 It is well-known that at least one such prime filtration always exists. Indeed, let $P\in\Ass(M)$. Then $P$ is a monomial prime ideal and there exists a homogeneous element $u\in M$, say of degree $a\in\ZZ^n$,  such that $uS\iso (S/P)(-a)$, cf.\ \cite[Lemma 1.5.6]{BrHe}. We set $M_1=uS$, and apply the same reasoning to $M/M_1$. Noetherian induction completes the proof.

 Each prime filtration ${\mathcal F}$ of $M$ gives rise to a Stanley decomposition ${\mathcal D}({\mathcal F})$ as follows:
 Since $M_i/M_{i-1}\iso S/P_i(-a_i)$, there exists a homogeneous element $u_i\in M_i$ of degree $a_i$, whose residue class modulo $M_{i-1}$ generates $M_i/M_{i-1}$ and such that $u_iK[Z_i]\iso M_i/M_{i-1}$,  where $Z_i=\{x_j\:\;x_j\not\in P_i\}$ and where $u_iK[Z_i]$ is a free $K[Z_i]$-module. The filtration $\mathcal F$ provides a decomposition $M=\Dirsum_{i=1}^mM_i/M_{i-1}$ as direct sum of $K$-vector spaces. Since each of the factors $M_i/M_{i-1}$ is a Stanley space $u_iK[Z_i]$ we obtain the decomposition  ${\mathcal D}({\mathcal F})=\Dirsum_{i=1}^mu_iK[Z_i]$, as desired. We say that ${\mathcal D}({\mathcal F})$ is the Stanley decomposition induced by the prime filtration $\mathcal F$.

 Not all Stanley decompositions of $M$ are induced by prime filtrations. Indeed,  a prime filtration $\mathcal F$ of $M$ is essentially the same thing as a sequence of homogeneous  generators $u_1,\ldots, u_m$ of $M$ with the property that the colon ideal $(u_1,\ldots, u_{i-1}):u_i$ is generated by a subset of the variables, in other words, is a monomial prime ideal, say $P_i$. We call such a sequence in $M$ a {\em sequence with linear quotients}. We say that $M$ has {\em linear quotients} if there exists a minimal set of homogeneous generators of $M$ which is a sequence with linear quotients. From our discussions  so far it follows that a Stanley decomposition ${\mathcal D}\: M=\Dirsum_{i=1}^mu_iK[Z_i]$ is induced by a prime filtration of $M$ if and only if, after  a suitable renumbering of the direct summands,  we have $(u_1,\ldots, u_{i-1}):u_i=(x_j\:x_j\not\in Z_i)$.

Consider for example the Stanley decomposition of the ideal
\begin{eqnarray}
\label{decm}
(x_1,x_2,x_3)= x_1x_2x_3K[x_1,x_2,x_3]\dirsum x_1K[x_1,x_2]\dirsum x_2K[x_2,x_3]\dirsum x_3K[x_1,x_3].
\end{eqnarray}
This Stanley decomposition is not induced by a prime filtration of $(x_1,x_2,x_3)$. In fact,  no order of the elements  $x_1x_2x_3, x_1, x_2, x_3$  is a sequence with linear quotients.

For later applications we will give the following simple characterization of Stanley decompositions induced by  a prime filtration.

\begin{Proposition}
\label{induced}
Let $M$ be a finitely generated $\ZZ^n$-graded $S$-module and ${\mathcal D}: M=\Dirsum _{i=1}^m u_iK[Z_i]$ a Stanley decomposition of $M$. Then the following conditions are equivalent:
\item[ (a)] $\mathcal D$ is induced by a prime filtration.
\item[ (b)] After a suitable relabeling of the summands in $\mathcal D$ we have $M_j=\Dirsum_{i=1}^ju_iK[Z_i]$ is a $\ZZ^n$-graded submodule of $M$ for $j=1,\ldots, m$.
\end{Proposition}

\begin{proof}
(a) \implies (b) follows immediately from the construction of a Stanley decomposition which is induced by a prime filtration.

(b)\implies (a): We claim that $\mathcal F\: 0\subset M_1\subset M_2\subset \ldots \subset M_m=M$ is a prime filtration  of $M$. First notice that for each $j$, the module $M_j/M_{j-1}$ is a cyclic module generated by the residue class $\bar{u}_j=u_j+M_{j-1}$. Indeed, each element $u\in M_j$ can be written as  $u=\sum_{k=1}^j u_kf_k$ with $f_k\in K[Z_k]$ for $k=1,\ldots,j$. Therefore $\bar{u}=\bar{u}_jf_j$.

Next we claim that  the annihilator of $\bar{u}_j$ is equal to the monomial prime ideal $P$ generated by  the variables $x_k\not\in Z_j$. In fact, if $x_k\not \in Z_j$, then $\deg x_ku_j\neq \deg u_jv$ for all monomials $v\in K[Z_j]$. Therefore, since $M_j=\Dirsum_{i=1}^ju_iK[Z_i]$ is a decomposition of $\ZZ^n$-graded $K$-vector spaces, it follows that $x_ku_j\in M_{j-1}$. This implies that $x_k\bar{u}_j=0$ and shows that $P$ is contained in the annihilator of $\bar{u}_j$. On the other hand, if $v$ is a monomial in $S\setminus P$, then $v\in K[Z_j]$ and so $u_jv$ is a nonzero element in $ u_jK[Z_j]$. This implies that $v$ does not belong to the annihilator of $\bar{u}_j$ and shows that $P$ is precisely the annihilator of $\bar{u}_j$.
From all this  we conclude that ${\mathcal D}$ is induced by $\mathcal F$.
\end{proof}

\begin{Proposition}
\label{wellknown} Let $M$ be a finitely generated $\ZZ^n$-graded $S$-module, and let $\mathcal F$ be a prime filtration of $M$. Then
\[
\min\{\dim S/P\:\; P\in {\mathcal F}\}\leq \depth M, \sdepth M\leq \min\{\dim S/P\:\; P\in \Ass(M)\}.
\]
\end{Proposition}

\begin{proof}
The bounds for the depth are well-known. For the convenience of the reader we give the references. One has $\depth M\leq \dim S/P$ for all $P\in\Ass M$, see \cite[Proposition 1.2.13]{BrHe}. This gives the upper bound for the depth.

Let ${\mathcal F}\: 0=M_0\subset M_1\subset \cdots \subset M_m=M$ be the given prime filtration of $M$. The exact sequence $0\to M_1\to M\to M/M_1\to 0$ yields the inequality  $\depth M\geq  \min\{\depth M_1, \depth M/M_1\}$, see \cite[Proposition 1.2.9]{BrHe}. Therefore the  lower bound for the depth follows by induction on the length of the filtration.

The lower bound for  $\sdepth M$ is  due to the fact that any filtration induces a Stanley decomposition. The upper bound for $\sdepth M$ has been shown by Apel \cite{Ap1} in case that $M=S/I$ where $I$ is a monomial ideal. By the same reasoning one can show the upper bound for general $M$, see \cite{Ja1}.
\end{proof}

It is clear that whenever $\depth M$ attains the lower bound given in Proposition~\ref{wellknown}, then Stanley's conjecture holds for $M$. This situation happens of course if the upper and lower bound given in Proposition~\ref{wellknown} coincide. This is the case if $M$ admits a prime filtration $\mathcal F$ with $\supp({\mathcal F})= \Ass(M)$ in which case $M$ is said to be {\em almost clean}. According to Dress \cite{Dr} the module $M$ is called {\em clean}, if there exists a prime filtration with $\supp({\mathcal F})=\Min(M)$. The combinatorial significance of this notion is that the Stanley--Reisner ring $K[\Delta]$ of a simplicial complex is clean if and only if $\Delta$ is shellable, see \cite[Theorem]{Dr}. This result has been extended in \cite{HePo} to $K$-algebras $S/I$ where $I$ is a monomial ideal, not necessarily squarefree. This is achieved by introducing pretty clean modules. A $\ZZ^n$-graded $S$-module $M$ is called {\em pretty clean} if $M$ admits a filtration ${\mathcal F}\: 0=M_0\subset M_1\subset \cdots \subset M_m=M$ with $M_i/M_{i-1}\iso S/P_i$ and such that for all $i<j$ with $P_i\subset P_j$  it follows that $P_i=P_j$. It is easy to see that a pretty clean module is almost clean  (see \cite[Corollary 3.4]{HePo}), so that pretty clean modules satisfy Stanley's conjecture. In case $M=S/I$ where $I$ is a monomial ideal, the property of being pretty clean  is equivalent to say that the associated multicomplex is shellable, see \cite[Theorem 10.5]{HePo}. Thus we have the following implications:
\begin{center}
shellable $\iff$ clean $\implies$ pretty clean $\implies$ almost clean,
\end{center}
and each of these conditions implies that $\depth=\sdepth$. On the other hand, the inequalities in Proposition~\ref{wellknown} may all be strict. For example, let $M=\mm=(x_1,x_2,x_3)$ be the maximal ideal of $S=K[x_1,x_2,x_3]$ and $\mathcal F$ the prime filtration of $\mm$ corresponding to the sequence $x_1,x_2,x_3^2,x_3$ with linear quotients $0:x_1=0$, $(x_1):x_2=(x_1)$, $(x_1,x_2):x_3^2=(x_1,x_2)$ and $(x_1,x_2,x_3^2):x_3=(x_1,x_2,x_3)$.
Then
\[
\min\{\dim S/P\:\; P\in {\mathcal F}\}=0< \depth \mm=1< \sdepth \mm=2<\min\{\dim S/P\:\; P\in \Ass(\mm)\}=3.
\]
The only question is why the Stanley depth of $\mm$ is equal to $2$. To see this, we first observe that for a monomial  ideal $I\subset K[x_1,\ldots,x_n]$ we have  $\sdepth I=n$, if and only if $I$ is a principal ideal.  Indeed, if $I=(u)$, then $I=uK[x_1,\ldots,x_n]$ is a Stanley decomposition. On the other hand, if $I$ is not principal at least two Stanley spaces are needed to cover $I$. Obviously any two Stanley spaces of dimension $n$ intersect, so that one of the summands in the Stanley decomposition must have  dimension smaller than $n$.

Thus  we have $\sdepth \mm\leq 2$. Since (\ref{decm})
is a Stanley decomposition of $\mm$ of Stanley depth 2, we see that $\sdepth \mm=2$.

In our example, the prime filtration $\mathcal F$ was not very well chosen. If we replace $\mathcal F$ by the prime filtration $\mathcal{F}'$ which is induced by the sequence $x_1,x_2,x_3$, then
$\min\{\dim S/P\:\; P\in {\mathcal F}'\}=\depth \mm=1$. Thus $\fdepth \mm=\depth \mm$ in this case.

It is  clear  Stanley's conjecture holds for $M$ if $\fdepth M=\depth M$. In general however, we may have $\fdepth M<\depth M$ as the following result shows.

\begin{Proposition}
\label{fdepth}
Let $K$ be a field, $\Delta$ be a simplicial complex and $K[\Delta]$ its Stanley--Reisner ring. Suppose that $K[\Delta]$ is Cohen--Macaulay. Then $\fdepth K[\Delta]=\depth K[\Delta]$ if and only if $\Delta$ is shellable.
\end{Proposition}

\begin{proof}
We have  $\fdepth K[\Delta]=\depth K[\Delta]$ if and only if there exists a prime filtration $\mathcal{F}$ of $K[\Delta]$ with $\dim S/P\geq \depth K[\Delta]=\dim K[\Delta]$ for all $P\in\supp \mathcal{F}$. This is the case if and only if $\supp \mathcal{F}$ is equal to the set of minimal prime ideals of $I_\Delta$. By the theorem of Dress \cite{Dr} this condition is satisfied if and only if $\Delta$ is shellable.
\end{proof}

Assume $K[\Delta]$ is not necessarily  Cohen--Macaulay. In view of Proposition~\ref{fdepth} one may ask whether $\Delta$ is shellable in the non-pure sense, provided  $\fdepth K[\Delta]=\depth K[\Delta]$. Unfortunately this is not always the case as the following simple example shows: let $\Delta$ be the simplicial complex on the vertex set $[4]$ with facets $\{1,2\}$ and $\{3,4\}$. Then $I_{\Delta}=(x_1x_3, x_1x_4, x_2x_3, x_2x_4)$ and $\depth K[\Delta]=1$. Denote by $y_i$ the residue class of $x_i$ modulo $I_\Delta$. The sequence  $y_1,y_3,y_4,1$ with linear quotients $0:y_1=(y_3,y_4)$, $(y_1):y_3=(y_1,y_2,y_4)$, $(y_1,y_2):y_4=(y_1,y_3)$ and $(y_1,y_2,y_4):1=(y_1,y_3,y_4)$ shows that $\fdepth K[\Delta]\geq 1$. Since, on the other hand, one always has $\fdepth K[\Delta]\leq \depth K[\Delta]$, we see that $\fdepth K[\Delta]=\depth K[\Delta]=1$. However, $\Delta$ is not  shellable.

\section{Stanley decompositions and partitions}
Let   $I\subset S$ be a monomial ideal. In this section we want to show that the Stanley depth of $I$ and of $S/I$ can be determined in a finite number of steps. In order to treat both cases simultaneously we will show this more generally for $\ZZ^n$-graded modules of the form $I/J$ where $J\subset I$ are monomial ideals in $S$.

We define a natural partial order on $\NN^n$ as follows: $a\leq b$ if and only if $a(i)\leq b(i)$ for $i=1,\ldots,n$. Note that $x^a|x^b$ if and only if $a\leq b$. Here, for any $c\in \NN^n$ we denote as usual by  $x^c$ the monomial $x_1^{c(1)}x_2^{c(2)}\cdots x_n^{c(n)}$. Observe that $\NN^n$ with the partial order introduced  is a distributive lattice with meet $a\wedge  b$  and join $a\vee b$ defined as follows: $(a\wedge  b)(i)=\min\{a(i),b(i)\}$ and $(a\vee  b)(i)=\max\{a(i),b(i)\}$. We also denote by $\epsilon_j$ the $j$th canonical unit vector in $\ZZ^n$.

Suppose $I$ is generated by the monomials $x^{a_1},\ldots, x^{a_r}$ and $J$ by the monomials $x^{b_1},\ldots, x^{b_s}$. We choose  $g\in\NN^n$ such that $a_i\leq g$ and $b_j\leq g$ for all $i$ and $j$, and let
$P^g_{I/J}$ be the set of all $c\in \NN^n$ with $c\leq g$ and such that $a_i\leq c$  for some $i$ and $c\not\geq b_j$ for all $j$. The  set $P^g_{I/J}$ viewed as a subposet of $\NN^n$ is a finite poset. We call it the {\em characteristic poset} of $I/J$ with respect to $g$. There is a natural choice for $g$, namely the join of all the $a_i$ and $b_j$. For this $g$, the poset $P_{I/J}^g$ has the least number of elements, and we denote it simply by $P_{I/J}$.
Note that if $\Delta$ is a simplicial complex on the vertex set $[n]$, then $P_{S/I_\Delta}$ is just the face poset of $\Delta$.

Figure~\ref{Fig2} shows the characteristic poset for the maximal ideal $\mm=(x_1,x_2,x_3)\subset K[x_1,x_2,x_3]$. The elements of this poset correspond to the squarefree monomials $x_1$, $x_2$, $x_3$, $x_1x_2$, $x_1x_3$, $x_2x_3$ and $x_1x_2x_3$. Thus the corresponding labels in Figure~\ref{Fig2} should be $(1,0,0), (0,1,0),\ldots, (1,1,1)$. In the squarefree case, like in this example, it is however more convenient and shorter to replace the  $(0,1)$-vectors (which  label the vertices in the characteristic poset)  by their support. In other words, each $(0,1)$-vector with support $\{i_1<i_2<\cdots <i_k\}$ is replaced by  $i_1i_2\cdots i_k$, as done in Figure~\ref{Fig2}.

\begin{figure}[hbt]
\begin{center}
\psset{unit=1.2cm}
\begin{pspicture}(0,0.75)(4,3.5)

 %\psline(6.5,0.6)(6
\psline(1,2.96)(1,2.055)
\psline(2,2.96)(2,2.055)
\psline(3,2.96)(3,2.055)
\psline(1.972,2.972)(1.037,2.037)
\psline(2.972,2.972)(2.037,2.037)
\psline(1.035,2.974)(2.968,2.042)
\psline(2,1.96)(2,1.055)
\psline(2.975,1.974)(2.037,1.037)
\psline(1.035,1.974)(1.972,1.038)
 \rput(2,1){$\circ$}
 \rput(1,2){$\circ$}
 \rput(2,2){$\circ$}
 \rput(3,2){$\circ$}
 \rput(1,3){$\circ$}
 \rput(2,3){$\circ$}
 \rput(3,3){$\circ$}
\rput(1,3.35){$1$}
\rput(2,3.35){$2$}
\rput(3,3.35){$3$}
\rput(0.65,2){$12$}
\rput(1.65,2){$23$}
\rput(3.35,2){$13$}
\rput(2,0.6){$123$}
 \end{pspicture}
 \end{center}
\caption{}\label{Fig2}
\end{figure}

Given any poset $P$ and $a,b\in P$ we set $[a,b]=\{c\in P\:\; a\leq c\leq b\}$ and call $[a,b]$ an {\em interval}. Of course, $[a,b]\neq \emptyset$ if and only if $a\leq b$. Suppose $P$ is a finite poset. A {\em partition} of $P$ is a disjoint union
\[
\mathcal{P}\:\; P=\Union_{i=1}^r[a_i,b_i]
\]
of intervals.

Figure~\ref{Fig3} displays a partition of the poset given in Figure~\ref{Fig2}.  The framed regions in Figure~\ref{Fig3} indicate that  $P_\mm=[1,12]\union [2,23]\union [3,13]\union [123,123]$.

\begin{figure}[hbt]
\begin{center}
\psset{unit=1.2cm}
\begin{pspicture}(0,0.75)(4,3.5)

 %\psline(6.5,0.6)(6
\psline(1,2.96)(1,2.055)
\psline(2,2.96)(2,2.055)
\psline(3,2.96)(3,2.055)
\psline(1.972,2.972)(1.037,2.037)
\psline(2.972,2.972)(2.037,2.037)
\psline(1.035,2.974)(2.968,2.042)
\psline(2,1.96)(2,1.055)
\psline(2.975,1.974)(2.037,1.037)
\psline(1.035,1.974)(1.972,1.038)
 \rput(2,1){$\circ$}
 \rput(1,2){$\circ$}
 \rput(2,2){$\circ$}
 \rput(3,2){$\circ$}
 \rput(1,3){$\circ$}
 \rput(2,3){$\circ$}
 \rput(3,3){$\circ$}
\rput(1,3.35){$1$}
\rput(2,3.35){$2$}
\rput(3,3.35){$3$}
\rput(0.65,2){$12$}
\rput(1.65,2){$23$}
\rput(3.35,2){$13$}
\rput(2,0.6){$123$}
\psecurve[linewidth=1.5pt](1,1.8)(1.2,2.5)(1,3.2)(0.8,2.5)(1,1.8)(1.2,2.5)(1,3.2)(0.8,2.5)
\psecurve[linewidth=1.5pt](2,1.8)(2.2,2.5)(2,3.2)(1.8,2.5)(2,1.8)(2.2,2.5)(2,3.2)(1.8,2.5)
\psecurve[linewidth=1.5pt](3,1.8)(3.2,2.5)(3,3.2)(2.8,2.5)(3,1.8)(3.2,2.5)(3,3.2)(2.8,2.5)
\psecurve[linewidth=1.5pt](2,0.8)(2.5,1)(2,1.2)(1.5,1)(2,0.8)(2.5,1)(2,1.2)(1.5,1)
 \end{pspicture}
 \end{center}
\caption{}\label{Fig3}
\end{figure}

We will show that each partition of $P^g_{I/J}$ gives rise to a Stanley decomposition of $I/J$.

In order to describe the Stanley decomposition of $I/J$ coming from a partition of $P^g_{I/J}$ we shall need the following notation: for each $b\in P^g_{I/J}$, we set $Z_b=\{x_j\:\; b(j)=g(j)\}$. We also introduce the function
\[
\rho \:\; P^g_{I/J}\to \ZZ_{\geq 0},\quad c\mapsto \rho(c),
\]
where $\rho(c) =|\{j\:\; c(j)=g(j)\}|(=|Z_c|)$. We then have

\begin{Theorem}
\label{partition}
Let  $\mathcal{P}\: P_{I/J}^g=\Union_{i=1}^r [c_i,d_i]$ be a partition of $P_{I/J}^g$. Then
\begin{eqnarray}
\label{dec}
\mathcal{D}(\mathcal{P})\: I/J=\Dirsum_{i=1}^r(\Dirsum_{c} x^cK[Z_{d_i}])
\end{eqnarray}
is a Stanley decomposition of $I/J$, where the inner direct sum is taken over all $c\in [c_i,d_i]$ for which   $c(j)=c_i(j)$  for  all  $j$ with $x_j\in Z_{d_i}$.
Moreover, $\sdepth \mathcal{D}(\mathcal{P})=\min\{\rho(d_i)\:\; i=1,\ldots,r\}$.
\end{Theorem}

\begin{proof}
We first show that the sum of the $K$-vector spaces in (\ref{dec}) is equal to the $K$-vector space  spanned by all monomials $u\in I\setminus J$ (which of course is isomorphic to the $K$-vector space  $I/J$).

Let $u=x^e$ be a monomial in $I\setminus J$ and let $c'=e\wedge g$. Then, $c'\in P_{I/J}^g$ and consequently, there exists $i\in \{1,\ldots,r\}$ such that $c'\in [c_i,d_i]$. Let  $c$ be the vector with
\[
c(j)=\left\{
\begin{array}{ll}
c_i(j), &\text{if $d_i(j)=g(j)$},\\
c'(j), & \text{otherwise}.\\
\end{array}
\right.
\]
It follows from the definition of $c$ that $x^{c}K[Z_{d_i}]$ is one of the Stanley spaces appearing in  (\ref{dec}).  We claim   that  $u\in x^{c}K[Z_{d_i}]$, equivalently, that  $x^{e-c}\in K[Z_{d_i}]$. Indeed, if $x_j\in Z_{d_i}$, then  $d_i(j)=g(j)$, and hence $e(j)\geq c'(j)\geq c_i(j)=c(j)$. On the other hand, if $x_j\notin Z_{d_i}$,  then   $g(j)>d_i(j)\geq c'(j)=c(j)$. Since $c'(j)=\min\{e(j),g(j)\}$, it therefore follows that $e(j)=c(j)$, as desired.

In order to prove that the sum (\ref{dec}) is direct, it suffices to show that any two different  Stanley spaces in (\ref{dec}) have no monomial in common. Suppose to the contrary  that $x^b\in x^pK[Z_{d_i}]\cap x^qK[Z_{d_j}]$ and that $x^pK[Z_{d_i}]\neq  x^qK[Z_{d_j}]$ are both summands in (\ref{dec}). Since each of the inner sums in (\ref{dec}) is direct, we have  that $i\neq j$.

We claim that  $x^b\in x^pK[Z_{d_i}]$ yields $b\wedge g\in [c_i,d_i]$.  
Indeed, since $c_i\leq b\wedge g$,  the claim follows once it is shown that  $b\wedge g\leq d_i$.   If $d_i(j)=g(j)$, then
\[
(b\wedge g)(j)=\min\{b(j),g(j)\}\leq g(j)=d_i(j).
\]
If $d_i(j)<g(j)$, then $x_j\notin Z_{d_i}$ and hence  $b(j)=p(j)$.  Together with the inequality $p(j)\leq d_i(j)<g(j)$, we obtain that $(b\wedge g)(j)=p(j)\leq d_i(j)$. In both cases the claim follows. 

Similarly, since $x^b\in x^qK[Z_{d_j}]$ we see that  $b\wedge g\in [c_j,d_j]$. This is a contradiction, since $[c_i,d_i]\sect [c_j,d_j]=\emptyset$.

The statement about the Stanley depth of $\mathcal{D}(\mathcal{P})$ follows immediately from the the definitions.
\end{proof}

We consider two examples to illustrate Theorem~\ref{partition}. As  a first example, consider the partition of the poset $P_\mm$ given in Figure~\ref{Fig3}. According to Theorem~\ref{partition} the Stanley decomposition corresponding to this partition is exactly the one given in (\ref{decm}).

The second, slightly more involved example, is displayed in Figure~\ref{Fig4}. In the first picture the hatched region corresponds to the $K$-vector space spanned by all monomials in $I\setminus J$ where  $I=(x_1^2x_2^4, x_1^3x_2^3, x_1^5x_2)$ and $J=(x_1^4x_2^5,x_1^6x_2^2)$. The second picture shows a partition of $P^g_{I/J}$ where $g=(7,6)$. The partition is the following:
\[
P^g_{I/J}=[(2,4),(3,6)]\union [(4,3),(5,4)]\union [(5,1),(7,1)]\union [(3,3),(3,3)]\union [(5,2),(5,2)].
\]
To this partition corresponds by Theorem~\ref{partition} the following Stanley decomposition
\[
I/J=(x_1^2x_2^4K[x_2]\dirsum x_1^3x_2^4K[x_2])\dirsum (x_1^4x_2^3K\dirsum x_1^5x_2^3K\dirsum x_1^4x_2^4K\dirsum x_1^5x_2^4K)\dirsum x_1^5x_2K[x_1]\dirsum x_1^3x_2^3K\dirsum  x_1^5x_2^2K
\]
which is shown in the third picture of Figure~\ref{Fig4}.

\begin{figure}[hbt]
\begin{center}
\psset{unit=0.6cm}
\begin{pspicture}(-1,-9)(19,7)
\psline(0,0)(8,0)
\psline(0,0)(0,7)
\psline[linewidth=1.1pt](5,1)(8,1)
\psline(6,2)(8,2)
\psline[linewidth=1.1pt](5,1)(5,3)
\psline(6,2)(6,5)
\psline[linewidth=1.1pt](3,3)(5,3)
\psline(4,5)(6,5)
\psline[linewidth=1.1pt](3,3)(3,4)
\psline[linewidth=1.1pt](2,4)(3,4)
\psline[linewidth=1.1pt](2,4)(2,7)
\psline(4,5)(4,7)
 \rput(5,1){$\bullet$}
 \rput(6,2){$\bullet$}
 \rput(3,3){$\bullet$}
 \rput(2,4){$\bullet$}
 \rput(4,5){$\bullet$}
\rput(4.7,5.45){$(4,5)$}
\rput(6.7,2.45){$(6,2)$}
\rput(1.4,3.5){$(2,4)$}
\rput(2.4,2.5){$(3,3)$}
\rput(4.4,0.5){$(5,1)$}
\pspolygon[fillstyle=vlines,hatchangle=45,fillcolor=black,linestyle=none](2,7)(2,4)(3,4)(3,3)(5,3)(5,1)(8,1)(8,2)(6,2)(6,5)(4,5)(4,7)
\psline(10,0)(18,0)
\psline(10,0)(10,7)
\psline(15,1)(18,1)
\psline(16,2)(18,2)
\psline(15,1)(15,3)
\psline(16,2)(16,5)
\psline(13,3)(15,3)
\psline(14,5)(16,5)
\psline(13,3)(13,4)
\psline(12,4)(13,4)
\psline(12,4)(12,7)
\psline(14,5)(14,7)
 \rput(15,1){$\bullet$}
 \rput(16,1){$\bullet$}
 \rput(17,1){$\bullet$}
 \rput(15,2){$\bullet$}
 \rput(15,3){$\bullet$}
 \rput(15,4){$\bullet$}
 \rput(14,4){$\bullet$}
 \rput(14,3){$\bullet$}
 \rput(13,3){$\bullet$}
 \rput(13,4){$\bullet$}
 \rput(12,4){$\bullet$}
 \rput(13,5){$\bullet$}
 \rput(12,5){$\bullet$}
 \rput(13,6){$\bullet$}
 \rput(12,6){$\bullet$}
 \rput(18.1,6.4){$g=(7,6)$}
 \psline[linestyle=dashed](10,6)(17,6)
\psline[linestyle=dashed](17,0)(17,6)
 \psecurve[linewidth=1.1pt](12.5,3.8)(13.2,4)(13.35,5)(13.2,6)(12.5,6.2)(11.8,6)(11.65,5)(11.8,4)(12.5,3.8)(13.2,4)(13.35,5)(13.2,6)(12.5,6.2)(11.8,6)(11.65,5)(11.8,4)
\psecurve[linewidth=1.1pt](13,2.8)(13.2,3)(13,3.2)(12.8,3)(13,2.8)(13.2,3)(13,3.2)(12.8,3)
\psecurve[linewidth=1.1pt](14.5,2.8)(15.15,2.9)(15.2,3.5)(15.15,4.1)(14.5,4.2)(13.85,4.1)(13.8,3.5)(13.85,2.9)(14.5,2.8)(15.15,2.9)(15.2,3.5)(15.15,4.1)(14.5,4.2)(13.85,4.1)(13.8,3.5)(13.85,2.9)
\psecurve[linewidth=1.1pt](15,1.8)(15.2,2)(15,2.2)(14.8,2)(15,1.8)(15.2,2)(15,2.2)(14.8,2)
\psecurve[linewidth=1.1pt](16,0.71)(17.2,1)(16,1.29)(14.8,1)(16,0.71)(17.2,1)(16,1.29)(14.8,1)
\psline(5,-9)(13,-9)
\psline(5,-9)(5,-2)
\psline[linewidth=1.2pt](10,-8)(13,-8)
\psline[linewidth=1.2pt](7,-5)(7,-2)
\psline[linewidth=1.2pt](8,-5)(8,-2)
\psline[linestyle=dashed](11,-7)(13,-7)
\psline[linestyle=dashed](10,-8)(10,-6)
\psline[linestyle=dashed](11,-7)(11,-4)
\psline[linestyle=dashed](8,-6)(10,-6)
\psline[linestyle=dashed](9,-4)(11,-4)
\psline[linestyle=dashed](8,-6)(8,-5)
\psline[linestyle=dashed](7,-5)(8,-5)
\psline[linestyle=dashed](9,-4)(9,-2)
\rput(10,-7){$\bullet$}
 \rput(10,-6){$\bullet$}
 \rput(10,-5){$\bullet$}
 \rput(9,-5){$\bullet$}
 \rput(9,-6){$\bullet$}
 \rput(8,-6){$\bullet$}
 \rput(10,-8){$\bullet$}
 \rput(7,-5){$\bullet$}
 \rput(8,-5){$\bullet$}
 \end{pspicture}
 \end{center}
\caption{}\label{Fig4}
\end{figure}

The next result clarifies for which partitions $\mathcal{P}$ of $P^g_{I/J}$ the Stanley decomposition $\mathcal{D(P)}$ of $I/J$ is induced by a prime filtration.

\begin{Theorem}
\label{primepartition}
Let  $\mathcal{P}\: P_{I/J}^g=\Union_{i=1}^r [c_i,d_i]$ be a partition of $P_{I/J}^g$ with the property that for all $j$ the union  $\Union_{i=1}^j [c_i,d_i]$ is a poset ideal of $P_{I/J}^g$. Then $\mathcal{D(P)}$ is induced by a prime filtration.
\end{Theorem}

\begin{proof}
The idea of the proof is to order the Stanley spaces appearing in $\mathcal{D(P)}$ such that any initial  sum of $\mathcal{D(P)}$  is a $\ZZ^n$-graded submodule of $I/J$. Then the assertion follows from Proposition~\ref{induced}. We choose a total order $\succ$ on the Stanley spaces such that
\[
x^cK[Z_{d_i}]\succ x^eK[Z_{d_j}]\quad \text{if $i<j$ or $i=j$ and $|c|\geq |e|$}.
\]
Now, with respect to this order, we arrange the summands in $\mathcal{D(P)}$ in decreasing order, and prove by induction on $t$ that the sum of the first $t$ summands is a $\ZZ^n$-graded submodule of $I/J$.

Let $x^cK[Z_{d_i}]$ be the first Stanley space  in $\mathcal{D(P)}$. It follows from the definition of $\succ$, that $i=1$ and $|c|$ is maximal among all $e$ such that $x^eK[Z_{d_1}]$ is a summand of $\mathcal{D(P)}$. Since $c\in [c_1,d_1]$ and $c(j)=c_1(j)$ for all $j$ with $x_j\in Z_{d_1}$, the maximality of  $|c|$    implies that $c(j)=d_1(j)$ for all $j$ with $x_j\notin Z_{d_1}$. Therefore $c$ is the following vector
\[
c(j)=\left\{
\begin{array}{ll}
c_1(j), &\text{if $x_j\in Z_{d_1}$},\\
d_1(j), & \text{otherwise}.\\
\end{array}
\right.
\]
In order to prove that $x^cK[Z_{d_1}]$ is a $\ZZ^n$-graded submodule of $I/J$ it is enough to check that $x_kx^c\in x^cK[Z_{d_1}]$ or $x_kx^c\in J$, for all $k=1,\ldots,n$. If $x_k\in Z_{d_1}$, then it is straightforward that $x_kx^c\in x^cK[Z_{d_1}]$. Otherwise, we have $c(k)=d_1(k)<g(k)$. Since $(c+\varepsilon_k)(k)=d_1(k)+1\leq g(k)$ it follows that $c+\varepsilon_k \geq c\geq c_1$, $c+\varepsilon_k\not\leq d_1$ and $c+\varepsilon_k\leq g$.  If $c+\varepsilon_k\in P_{I/J}^g$ then, since $c+\varepsilon_k\geq c_1$ and $[c_1,d_1]$ is a poset ideal of $P_{I/J}^g$ we obtain that $c+\varepsilon_k\in [c_1,d_1]$, a contradiction. Therefore, $c+\varepsilon_k\not\in P_{I/J}^g$ and because $c+\varepsilon_k\leq g$ we get that there exists $j$ such that $b_j\leq c+\varepsilon_k$. Hence $x^{c+\varepsilon_k}=x_kx^c\in J$, and we are done.

For the induction step, assume that the sum of the first $t$ Stanley spaces, say $M$, with $t\geq 1$, is a $\ZZ^n$-graded submodule of $I/J$, and we need to show that the sum of the first $t+1$ Stanley spaces is again a $\ZZ^n$-graded submodule of $I/J$. We may assume that the $(t+1)$-th Stanley space is $x^{c'}K[Z_{d_l}]$, with $1\leq l\leq r$ and $c'\in [c_l,d_l]$, such that $c'(j)=c_l(j)$, for all $j$ with $x_j\in Z_{d_l}$. In order to prove that $M\dirsum x^{c'}K[Z_{d_l}]$ is a $\ZZ^n$-graded submodule of $I/J$ it is enough to check that $x_kx^{c'}\in M\dirsum x^{c'}K[Z_{d_l}]$ or $x_kx^{c'}\in J$ for all $k=1,\ldots,n$. If $x_k\in Z_{d_l}$, then obviously $x_kx^{c'}\in x^{c'}K[Z_{d_l}]$, hence $x_kx^{c'}\in M\dirsum x^{c'}K[Z_{d_l}]$. Otherwise, we have $c'(k)\leq d_l(k)<g(k)$. We have to consider  two cases.

Case 1: $c'(k)<d_l(k)$. Then, $c_l\leq c'<c'+\varepsilon_k\leq d_l$. Hence $c'+\varepsilon_k\in [c_l,d_l]$. Since $|c'+\varepsilon_k|=|c'|+1$, the order given for the Stanley spaces appearing in $\mathcal{D(P)}$ implies $x^{c'+\varepsilon_k}K[Z_{d_l}]\succ x^{c'}K[Z_{d_l}]$. Therefore, $x^{c'+\varepsilon_k}K[Z_{d_l}]\subset M$ and consequently $x_kx^{c'}=x^{c'+\varepsilon_k}\in M$.

Case 2: $c'(k)=d_l(k)$. Then,  $c_l<c'+\varepsilon_k\leq g$ and $c'+\varepsilon_k\not\leq d_l$. If $c'+\varepsilon_k\in P_{I/J}^g$, then since $\Union_{i=1}^l [c_i,d_i]$ is a poset ideal of $P_{I/J}^g$ and $c'+\varepsilon_k>c_l$ we obtain that $c'+\varepsilon_k\in\Union_{i=1}^l [c_i,d_i]$. On the other hand, $c'+\varepsilon_k\not\leq d_l$ implies $c'+\varepsilon_k\notin [c_l,d_l]$, hence $c'+\varepsilon_k\in\Union_{i=1}^{l-1} [c_i,d_i]$. Therefore $x_kx^{c'}=x^{c'+\varepsilon_k}\in M$. If $c'+\varepsilon_k\notin P_{I/J}^g$, then since $c_l<c'+\varepsilon_k\leq g$ we necessarily have that there exists $j$ such that $b_j\leq c'+\varepsilon_k$. This implies that $x_kx^{c'}=x^{c'+\varepsilon_k}\in J$, as desired.
\end{proof}

The preceding result can be used to compute the Krull dimension of $I/J$.

\begin{Corollary}
\label{krulldim}
$\dim I/J=\max\{\rho(c)\:\; c\in P^g_{I/J}\}$.
\end{Corollary}

\begin{proof}
Let $\mathcal F$ be any prime filtration of $I/J$. Then $\dim I/J=\max\{\dim S/P\:\; P\in \supp({\mathcal F})\}$. Now consider the  canonical partition ${\mathcal P}\:\; P_{I/J}^g=\Union_{c\in P_{I/J}^g}[c,c]$. We choose a total order $\succ$ of the intervals with the property that $[c,c]\succ [d,d]$ implies that $|d|\leq |c|$. Then the union of any initial sequence of these intervals is a poset ideal in $P_{I/J}^g$. Therefore it follows from Theorem~\ref{primepartition} that $\mathcal{D(P)}: I/J= \Dirsum_{c\in P^g_{I/J}}x^cK[Z_c]$ is induced by a prime filtration $\mathcal F$ of $I/J$.

It follows that
\begin{eqnarray*}
\dim I/J&=&\max\{\dim S/P\:\; P\in \supp({\mathcal F})\}\\
&=& \max\{|Z_c|\:\; c\in P^g_{I/J}\}= \max\{\rho(c)\:\; c\in P^g_{I/J}\}.
\end{eqnarray*}
\end{proof}

We are interested in computing the $\sdepth$ of a graded module $M$.  In the case that $M=I/J$, the next theorem shows that $\sdepth I/J$ can be computed in a finite number of steps.

\begin{Theorem}
\label{sdepth}
Let $\mathcal{D}$ be a Stanley decomposition of $I/J$. Then, there exists a partition $\mathcal{P}$ of $P_{I/J}^g$ such that
\[
\sdepth \mathcal{D}(\mathcal{P})\geq \sdepth \mathcal{D}.
\]
In particular, $\sdepth(I/J)$ can be computed as the maximum of the numbers $\sdepth \mathcal{D(P)}$, where $\mathcal{P}$ runs over the (finitely many)  partitions of $P_{I/J}^g$.
\end{Theorem}
\begin{proof}
Let $\mathcal{D}$ be an arbitrary Stanley decomposition of $I/J$. First, to each $b\in P_{I/J}^g$ we assign an interval $[c,d]\subset P_{I/J}^g$: since $x^b\in I\setminus J$, there exists a Stanley space $x^cK[Z]$ in the decomposition $\mathcal{D}$ of $I/J$ with  $x^b\in x^cK[Z]$. It follows that $c\in P_{I/J}^g$ and $b(j)=c(j)$ for all $j$ with $x_j\notin Z$. Now, we define $d\in \NN^n$ by setting
\[
d(j)=\left\{\begin{array}{ll}
g(j), &\text{if $x_j\in Z$},\\
c(j), & \text{if $x_j\notin Z$}.\\
\end{array}
\right.
\]
Observe that $[c,d]\subset   P_{I/J}^g$.  We noticed already that $c\in  P_{I/J}^g$. It remains to be shown that $d\in P_{I/J}^g$. Since $x^cK[Z]\in I\setminus J$, it follows that $x^{c+\sum_jn_j\epsilon_j}\in I\setminus J$, where the sum is taken over all $j$ with $x_j\in Z$ and where for all $j$ we have $n_j\in \ZZ_{\geq 0}$. Therefore $d= c+\sum_j(g(j)-c(j))\epsilon_j\in P_{I/J}^g$.

Next we show that $b\in [c,d]$. For this we need to show that $b\leq d$. Indeed, if $x_j\in Z$, then $b(j)\leq g(j)=d(j)$. Otherwise $d(j)=c(j)=b(j)$ and consequently the inequality holds. Since $b\in [c,d]$, we  obtain that $x^b\in x^cK[Z_d]$, and $Z\subseteq Z_d$, according to the definition of $d$.

In order to complete the proof of our theorem we now show that the intervals constructed above provide a partition $\mathcal P$ of $P^g_{I/J}$ and that $\sdepth \mathcal{D}(\mathcal{P})\geq\sdepth \mathcal{D}$.

It is clear that these intervals cover $P^g_{I/J}$. Therefore it is enough to check that for any $b_1,b_2\in P_{I/J}^g$ with $b_1\neq b_2$, the corresponding intervals obtained from our construction, say $[c_1,d_1]$ and  $[c_2,d_2]$, satisfy either $[c_1,d_1]=[c_2,d_2]$ or $[c_1,d_1]\cap [c_2,d_2]=\emptyset$.

To each $c_i$ corresponds a Stanley space $x^{c_i}K[Z_i]$ in the given Stanley decomposition $\mathcal D$. We consider two cases. In the first case, we assume that  $c_1=c_2$. Then  $Z_1=Z_2$, and consequently $d_1=d_2$. Hence $[c_1,d_1]=[c_2,d_2]$. In the second case, we assume $c_1\neq c_2$.
In this case we prove that $[c_1,d_1]\cap [c_2,d_2]=\emptyset$. Assume, by contradiction, that there exists $e\in P_{I/J}^g$ such that $e\in [c_1,d_1]\cap [c_2,d_2]$. It follows from the construction of the interval $[c_1,d_1]$ that $c_1(j)=d_1(j)$ if $x_j\notin Z_1$. Therefore, $e\in [c_1,d_1]$ implies that $e(j)=c_1(j)$, for all $j$ with $x_j\notin Z_1$, and hence we obtain that $x^e\in x^{c_1}K[Z_1]$. Analogously, one obtains that $x^e\in x^{c_2}K[Z_2]$, a contradiction since  $x^{c_1}K[Z_1]\cap x^{c_2}K[Z_2]=0$.

To establish now the inequality $\sdepth \mathcal{D(P)}\geq \sdepth \mathcal{D}$, we observe that $\sdepth \mathcal{D(P)}$ is equal to the minimum of all integers $|Z_d|$ where $[c,d]$ belongs to $\mathcal P$. On the other hand, we already showed that for each Stanley space $x^cK[Z]$ in $\mathcal D$ such that $c\in P_{I/J}^g$ we have that $|Z_d|\geq |Z|$. This yields the desired inequality.
\end{proof}

The following example demonstrates  the construction given in Theorem \ref{sdepth}. Let $I=(x_1^2,x_2^2)\subset K[x_1,x_2]$, then  $P_I=\{(2,0),(0,2),(2,1),(1,2),(2,2)\}$. Consider the following Stanley decomposition
\[
\mathcal{D}\: I=x_1^2K[x_1]\dirsum x_1^2x_2K[x_1,x_2]\dirsum x_1x_2^2K[x_2]\dirsum x_2^2K\dirsum x_2^3K[x_2]
\]
with $\sdepth(\mathcal{D})=0$. We apply the construction given in  the proof of the Theorem~\ref{sdepth} and  label the elements of $P_I$ in the order as listed above, with $b_1,\ldots,b_5$. To each $b_i$ we associate an interval $[c_i,d_i]$ as described in the proof of Theorem~\ref{sdepth}. For example,  since $x^{b_5}\in x_1^2x_2K[x_1,x_2]$, we obtain $c_5=(2,1)$ and $d_5=(2,2)$. Similarly, we obtain the intervals $[c_4,d_4]=[(1,2),(1,2)]$, $[c_3,d_3]=[(2,1),(2,2)]$, $[c_2,d_2]=[(0,2),(0,2)]$ and $[c_1,d_1]=[(2,0),(2,0)]$. We notice that $[c_5,d_5]=[c_3,d_3]$ and that
$
\Union_{i=1}^4 [c_i,d_i]
$
is a partition of $P_I$ which, according to Theorem \ref{partition}, gives the following Stanley decomposition
\[
\mathcal{D(P)}\: I=x_1^2K[x_1]\dirsum x_2^2K[x_2]\dirsum x_1^2x_2K[x_1,x_2]\dirsum x_1x_2^2K[x_2]
\]
with $\sdepth(\mathcal{P})=1$. In general the theorem asserts that $\sdepth \mathcal{D(P)}\geq \sdepth \mathcal{P}$. The example shows that it may  indeed be bigger.

\medskip
As a consequence of Theorem~\ref{partition} and Theorem~\ref{sdepth} we have

\begin{Corollary}
\label{finite}
Let $J\subset I$ be monomial ideals. Then
\[
\sdepth I/J=\max\{\sdepth \mathcal{D(P)}\:\;  \text{$\mathcal P$ is a partition of $P_{I/J}^g$} \}.
\]
In particular, there exists a partition $\mathcal{P}\:\; P_{I/J}^g=\Union_{i=1}^r[c_i,d_i]$  of $P_{I/J}^g$ such that
$$\sdepth I/J=\min\{\rho(d_i)\:\; i=1,\ldots,r\}.$$
\end{Corollary}

If we want to use Corollary~\ref{finite} in concrete cases to compute the Stanley depth, it is advisable to choose $g$ such that the poset  $P^g_{I/J}$ is as small as possible. If $G(I)=\{x^{a_1},\ldots, x^{a_r}\}$ and $G(J)=\{x^{b_1},\ldots, x^{b_s}\}$, then  with $g=a_1\vee \cdots \vee a_r\vee b_1\vee \cdots \vee b_s$ the poset  $P^g_{I/J}$ has the least number of elements.

The following examples demonstrate the power of Corollary~\ref{finite} and also show that in general it is very hard to compute the Stanley depth of a monomial ideal, even though it can be done in a finite number of steps.

\begin{Examples}
\label{maximal}{\em  Let $\mm$ be the graded maximal ideal of $S=K[x_1,\ldots,x_n]$. Then $\sdepth \mm=\lceil n/2\rceil $ for $n\leq 9$, where  $\lceil n/2\rceil $ denotes the smallest integer $\geq n/2$. We expect this to be true for all integers $n$, but do not have a general proof yet. Here we give a proof for $n=4$ and $5$ to demonstrate the kind of arguments we use. We use the same notation as used in  Figure~\ref{Fig2} where a set $\{i_1<i_2<\cdots <i_k\}$ is written as $i_1i_2\cdots i_k$.

(a) Let $n=4$. Then $P_{\mm}$ is the following collection of subsets of the set $1234$
\[
1 \quad 2 \quad 3 \quad 4
\]
\[
12 \quad 13 \quad 14 \quad 23 \quad 24 \quad 34
\]
\[
123 \quad 124 \quad 134 \quad 234
\]
\[
1234
\]
Let $A=[1,12]\cup [2,23]\cup [3,34]\cup [4,14]$. Then $A\cup\Union_{a\in P_{\mm}\setminus A}[a,a]$ is a partition of $P_{\mm}$ and by Corollary~\ref{finite} we obtain that $\sdepth \mm\geq 2$. On the other hand, since $\mm$ is not principal we have $\sdepth \mm\leq 3$. Assume that $\sdepth \mm=3$. By Corollary~\ref{finite} there exists a partition of $P_{\mm}$ into disjoint intervals such that the end point of each interval is at least a 3-set of the poset shown above. If one of these intervals is $[i,1234]$, say $[1,1234]$, then one   of the intervals $[2,234]$, $[3,234]$, $[4,234]$ would have to cover the rest, a contradiction. Otherwise we have four disjoint intervals of type $[i,ijk]$, where $1\leq i\leq 4$ and $ijk$ runs over the set $\{123,124,134,234\}$. Therefore the number of 2-sets in $P_{\mm}$  is at least  $4\times 2=8$, a contradiction. Hence, our assumption is false and consequently $\sdepth \mm=2=\lceil 4/2\rceil$.

(b) Let $n=5$. Obviously $A=[1,123]\cup [2,234]\cup [3,345]\cup [4,145]\cup [5,125]$ is a disjoint union of intervals which contains all 1- and 2-sets  of $P_{\mm}$. Then $A\cup \Union_{a\in P_{\mm}\setminus A}[a,a]$ is a partition of $P_{\mm}$ and applying Corollary~\ref{finite} we obtain that $\sdepth \mm\geq 3$. With the same arguments given in (a) one can show that $\sdepth \mm\neq 4$. Hence $\sdepth \mm=3=\lceil 5/2\rceil$.}
\end{Examples}

The next theorem and its corollary show that not only the $\sdepth$, but also the $\fdepth$ of $I/J$ can be computed in a finite number of steps.

\begin{Theorem}
\label{inducedprime}
Let $\mathcal{D}$ be  a Stanley decomposition of $I/J$ induced by  a prime filtration of $I/J$. Then there exists a partition $\mathcal P$ of $P^g_{I/J}$ with the property that  $\mathcal{D(P)}$ is induced by a prime filtration and such that $\sdepth \mathcal{D}(\mathcal{P})\geq \sdepth \mathcal{D}$.
\end{Theorem}

\begin{proof}
Let $\mathcal{D}:I/J=\Dirsum_{i=1}^t x^{c_i}K[Z_i]$ be a Stanley decomposition of $I/J$ induced by a prime filtration of $I/J$. Hence by Proposition~\ref{induced} we may assume that $\Dirsum_{i=1}^l x^{c_i}K[Z_i]$ is a $\ZZ^n$-graded submodule of $I/J$ for all $l$ with $1\leq l\leq t$. We shall prove that the
partition $\mathcal{P}$ of $P^g_{I/J}$ constructed in Theorem \ref{sdepth} satisfies the conditions of our theorem. Indeed, since $\sdepth \mathcal{D}(\mathcal{P})\geq \sdepth \mathcal{D}$ by Theorem \ref{sdepth}, it remains to be shown  that $\mathcal{D(P)}$ is induced by a prime filtration.

Let ${\mathcal S}$ be the subset of $\{1,2,\ldots,t\}$ with the property that $i\in \mathcal S$ if and only if there exists $e\in P^g_{I/J}$ such that $x^e\in x^{c_i}K[Z_i]$. Then by the  construction given in the proof of Theorem \ref{sdepth} there exists for each $i\in\mathcal S$ an element $d_i\in   P_{I/J}^g$  with $c_i\leq d_i$ and such that $P_{I/J}^g=\Union_{i\in \mathcal S} [c_{i},d_{i}]$    is a partition of $P_{I/J}^g$.  Moreover, $e\in P_{I/J}^g $ belongs to $[c_i,d_i]$ if $x^e\in x^{c_i}K[Z_i]$.  Say, $\mathcal{S}=\{i_1,i_2,\ldots, i_r\}$ with $i_1<i_2<\ldots <i_r$. We claim that $\Union_{j=1}^p [c_{i_j},d_{i_j}]$ is a poset ideal for all $p$ with $1\leq p\leq r$. Then Theorem~\ref{primepartition} implies that  $\mathcal{D(P)}$ is induced by a prime filtration, and we are done.

We prove our claim by induction on $p$. For $p=1$, let $e\in P^g_{I/J}$ such that $e\geq c_{i_1}$. Then we have $x^{c_{i_1}}|x^e$. Since $\mathcal{D}$ is induced by a prime filtration  of $I/J$ it follows that $x^e\in\Dirsum_{i=1}^{i_1} x^{c_i}K[Z_i]$. If $x^e\in x^{c_{i_1}}K[Z_{i_1}]$, then $e\in [c_{i_1},d_{i_1}]$, as desired. Otherwise $i_1>1$ and $x^e\in\Dirsum_{i=1}^{i_1-1} x^{c_i}K[Z_i]$. Therefore $x^e\in x^{c_i}K[Z_i]$ for some $i<i_1$. This implies that $i\in \mathcal S$,  a contradiction.

Now assume that  $p>1$ and that  $\Union_{j=1}^{p-1} [c_{i_j},d_{i_j}]$ is a poset ideal. It is enough to check that for any $e\in P_{I/J}^g$ with $e\geq c_{i_p}$ we have $e\in \Union_{j=1}^p [c_{i_j},d_{i_j}]$. Indeed, $e\geq c_{i_p}$ implies $x^{c_{i_p}}|x^e$ and consequently $x^e\in\Dirsum_{i=1}^{i_p} x^{c_i}K[Z_i]$, since $\mathcal{D}$ is induced by a prime filtration. If $x^e\in x^{c_{i_p}}K[Z_{i_p}]$, then $e\in [c_{i_p},d_{i_p}]$. Otherwise $x^e\in\Dirsum_{i=1}^{i_p-1} x^{c_i}K[Z_i]$ and  therefore   $x^e\in x^{c_i}K[Z_i]$ for some $i<i_p$. Hence  $i\in\{i_1,\ldots,i_{p-1}\}$, and consequently $e\in\Union_{j=1}^{p-1} [c_{i_j},d_{i_j}]$.
\end{proof}

As a consequence of Theorem~\ref{primepartition} and Theorem~\ref{inducedprime} we now obtain

\begin{Corollary}
\label{ffinite}
Let $J\subset I$ be monomial ideals. Then $\fdepth I/J$ is the maximum of the numbers $\sdepth \mathcal{D(P)}$, where the maximum is taken over all  partitions $\mathcal{P}=\Union_{i=1}^r[c_i,d_i]$ of $P_{I/J}^g$ with the property that
$\Union_{i=1}^j[c_i,d_i]$ is a poset ideal of $P_{I/J}^g$ for all $j$.
\end{Corollary}

\section{Applications and examples}

As shown in the previous section, the $\sdepth$ as well as the $\fdepth$ of $I/J$ for monomial ideals $J\subset I$ can be computed by considering the  partitions of the (finite) characteristic poset $P_{I/J}^g$. This does not mean that these invariants can be computed in practice, because the number of possible partitions can easily become very huge. In this section we will show that the techniques of the previous section nevertheless allow us to give bounds and in some cases even to compute these invariants.

The following proposition reassembles some observations we implicitly made in the previous sections.

\begin{Proposition}
\label{equal}
Let $J\subset I$ be  monomial ideals. Then
\begin{enumerate}
\item[(a)] $\fdepth S/I=\depth S/I$,  if $S/I$ is pretty clean;
\item[(b)] $\fdepth I= \depth I$,  if $I$ has linear quotients;
\item[(c)] $\fdepth I/J\geq \min\{\rho(c)\:\; c\in P_{I/J}\}$. In particular, if $I$ is a squarefree monomial ideal, then $\fdepth I\geq \min\{\deg u\:\; u\in G(I)\}$.
\end{enumerate}
\end{Proposition}

\begin{proof}
(a) Let $\mathcal F$ be a pretty clean filtration of $S/I$. As we mentioned already in Section 1, we have $\Ass(S/I)=\supp \mathcal{F}$. Thus it follows from Proposition~\ref{wellknown} that $\fdepth S/I=\depth S/I$.

(b) By assumption, $G(I)=\{u_1,\ldots, u_r\}$ and $P_i=(u_1,\ldots,u_{i-1})\: u_i$ is generated by a subset of $\{x_1,\ldots,x_n\}$ for each $i$. Let $m_i$ be the number of generators of $P_i$. It is shown  in \cite{HeTa} that $\projdim I=\max\{m_1,\ldots,m_r\}$, so that $\depth I=n-\max\{m_1,\ldots,m_r\}=\min\{n-m_1,\ldots,n-m_r\}$. On the other hand, ${\mathcal F}\:\;(0)\subset (u_1)\subset (u_1,u_2)\subset \ldots \subset I$ is a prime filtration
 of $I$ with $\supp \mathcal {F}=\{P_1,\ldots, P_r\}$. Hence $\fdepth I\geq \min\{\dim S/P_1,\ldots, \dim S/P_r\}=\depth I$. Since we always have  $\fdepth I\leq \depth I$, the assertion follows.

(c) We already observed in the proof of Corollary~\ref{krulldim} that ${\mathcal P}\: I/J=\Union_{c\in P_{I/J}}[c,c]$ induces a prime filtration $\mathcal F$. It follows from the definitions that $$\min\{\dim S/P\:\; P\in \supp {\mathcal F}\}= \min\{\rho(c)\:\; c\in P_{I/J}\}.$$ This yields the desired inequality. In the squarefree case, $\rho(c)=|c|=\deg x^c$. This implies the second part of statement (c).
\end{proof}

We would like to mention that Soleyman-Jahan \cite{Ja1} proved with the same arguments that $\sdepth I\geq \depth I$ if $I$ has linear quotients.

\medskip
As an example, consider the ideal $I_{n,d}$ generated by all squarefree monomials of degree $d$ in $n$ variables. $I_{n,d}$ is the Stanley--Reisner ideal  of the $(d-1)$-skeleton of the $n$-simplex. Since all the skeletons of the $n$-simplex are shellable, it follows from Proposition~\ref{equal}(a) and the discussions in Section~1 that  $\fdepth S/I_{n,d}=\sdepth S/I_{n,d}=\depth S/I_{n,d}=d-1$.

It is known \cite{CoHe} that $I_{n,d}$ has linear quotients since $I_{n,d}$ is a polymatroidal ideal. Therefore Proposition~\ref{equal}(b) implies that $\fdepth I_{n.d}=\depth I_{n,d}=d$. This fact one could also deduce from Proposition~\ref{equal}(c), since  all generators of $I_{n,d}$ are of degree $d$.

To compute  $\sdepth I_{n,d}$ is much harder. Even for the graded maximal ideal $\mm=I_{n,1}$, we cannot compute the Stanley depth in general, see Example~\ref{maximal}.

\begin{Proposition}
\label{complete}
Let $I\subset S$ be a monomial complete intersection. Then  $\fdepth S/I=\depth S/I$ and $\fdepth I=\depth I$. In particular, Stanley's conjecture holds for $S/I$ and $I$.
\end{Proposition}

\begin{proof}
The equality  $\fdepth S/I=\depth S/I$ follows from the fact that $S/I$ is pretty clean, as shown in \cite{HeJaYa}.

Let $G(I)=\{u_1,\ldots,u_r\}$.
In order to compute the $\fdepth$ of $I$ we consider the filtration
\[
(0)\subset (u_1)\subset (u_1,u_2)\subset\ldots  \subset (u_1,\ldots, u_r)=I.
\]
We have
\[
(u_1,\ldots, u_i)/(u_1,\ldots,u_{i-1})\iso S/(u_1,\ldots,u_{i-1}):u_i= S/(u_1,\ldots,u_{i-1}),
\]
for all $i$, since $u_1,\ldots, u_r$ is a regular sequence.  It follows that
\begin{eqnarray*}
\fdepth I&\geq &\min\{\fdepth S/(u_1,\ldots,u_i)\:\; i=1,\ldots,r-1\}\\
&=&\min\{\depth S/(u_1,\ldots,u_i)\:\; i=1,\ldots,r-1\}\\
&=&\depth S/(u_1,\ldots,u_{r-1})=n-r+1=\depth I.
\end{eqnarray*}
Therefore $\fdepth I=\depth I$.
\end{proof}

After these examples one might have the impression that one always has $\fdepth I=\depth I$. This is however not the case as the following example shows: let $\Delta$ be the simplicial complex on the vertex set $\{1,\ldots,6\}$, associated to a triangulation of the real projective plane $\PP^2$, whose facets are
\[
\mathcal F(\Delta)=\{125,126,134,136,145,234,235,246,356,456\}.
\]
Then the Stanley-Reisner ideal of $\Delta$ is
\[
I_{\Delta}=(x_1x_2x_3,x_1x_2x_4,x_1x_3x_5,x_1x_4x_6,x_1x_5x_6,x_2x_3x_6,x_2x_4x_5,x_2x_5x_6,x_3x_4x_5,x_3x_4x_6).
\]
It is known that $\depth I_{\Delta}=4$ if $\chara K\neq 2$ and $\depth I_{\Delta}=3$ if $\chara K=2$. Since the inequality $\fdepth I_{\Delta}\leq \sdepth I_{\Delta}$ holds independent of the characteristic of the base field, we obtain that $\fdepth I_{\Delta}\leq 3$. On the other hand it follows from the Proposition~\ref{equal}(c) that $\fdepth I_{\Delta}\geq 3$. Therefore $\fdepth I_{\Delta}=3$ and $\fdepth I_{\Delta}<\depth I_{\Delta}$ for any field $K$ with $\chara K\neq 2$.

\medskip
We now give  a lower bound for the sdepth of a monomial ideal by using a strategy  which is modeled after the Janet algorithm (see \cite{J} and \cite{PlRo}) and which allows to use induction on the number of variables. Let $I\subset S$ be a monomial ideal with $G(I)=\{x^{a_1},\ldots, x^{a_m}\}$. We set $a=a_1\vee a_2\vee \cdots \vee a_m$. Then we  can write $P_I$ as a disjoint union $P_I=\Union_{j=p}^q A_j$, where  $p=\min\{a_1(n),\ldots,a_m(n)\}$, $q=a(n)$ and $A_j=\{c\in P_I: c(n)=j\}$.  For all $j$ with $p\leq j\leq q$ we let $I_j$ be the monomial ideal of $K[x_1,\ldots,x_{n-1}]$ such that $I\cap {x_n}^jK[x_1,\ldots,x_{n-1}]={x_n}^jI_j$. Then for all $j$ with $p\leq j\leq q$, we have $A_j=\{(c,j): c\in P_{I_j}^g\}$ with $g=(a(1),\ldots,a(n-1))$.

\begin{Proposition}
\label{recursion}
With the notation introduced we have $$\sdepth I\geq \min\{\sdepth I_p, \ldots, \sdepth I_{q-1},\sdepth I_q+1\}.$$
\end{Proposition}

\begin{proof} By Corollary~\ref{finite} there exists  for each $j\in\{p,\ldots,  q\}$ a partition
$ P_{I_j}^g=\Union_{k=1}^{r_j}[c_{jk},d_{jk}]$  of
$P_{I_j}^g$ with $\sdepth I_j=\min\{\rho(d_{jk})\:\; k=1,\ldots,r_j\}$. Since $P_I$ is the disjoint union of the $A_j$ it follows that $P_I= \Union_{j=p}^q\Union_{k=1}^{r_j}[(c_{jk},j),(d_{jk},j)]$ is a partition of $P_I$. We have
\[
\rho(d_{jk},j)= \left\{ \begin{array}{ll}
       \rho(d_{jk}), & \;\text{if  $j<q$}, \\
        \rho(d_{jk})+1, & \;\text{if $j=q$}.
        \end{array} \right.
\]
Hence the conclusion follows from Theorem~\ref{partition} and Theorem~\ref{sdepth}.
\end{proof}

Now we are ready to prove

\begin{Proposition}
\label{lower} Let $I\subset S$ be a monomial ideal generated by $m$ elements. Then
\[
\sdepth I \geq \max\{1,n-m+1\}
\]
\end{Proposition}

\begin{proof} We may assume that $m$ is the number of minimal monomial generators of $I$. Then we proceed by induction on $n$. If $n=1$, then $I=(u)$ is a principal ideal with Stanley decomposition $I=uK[x_1]$. Therefore, $\sdepth I=1$. For the induction step we shall use Proposition \ref{recursion}. Indeed, we already have that $I_j$ is a monomial ideal of $K[x_1,\ldots,x_{n-1}]$ for all $j$, with $p\leq j\leq q$. In addition, one can easily see that $|G(I_j)|<m$ for all $j$ such that $j<q$, and $|G(I_q)|\leq m$. Hence, by induction hypothesis we have $\sdepth I_j\geq\max\{1,n-|G(I_j)|\}\geq\max\{1,n-m+1\}$ for all $j$ with $j<q$, and similarly the induction hypothesis  implies that $\sdepth I_q\geq\max\{1,n-m\}$, so that  $\sdepth I_q +1\geq\max\{2,n-m+1\}\geq\max\{1,n-m+1\}$. Applying now Proposition~\ref{recursion} we obtain the desired inequality.
\end{proof}

\begin{Corollary}
\label{two}
Let $I\subset S$ be a monomial ideal minimally  generated by 2 elements. Then $\fdepth I=\depth I=\sdepth I=n-1$.
\end{Corollary}

\begin{proof} Let $G(I)=\{u_1,u_2\}$. Then $I=v(v_1,v_2)$ where $v=\gcd(u_1,u_2)$ and $v_1,v_2$ is a regular sequence. It follows that, up to shift, the $\ZZ^n$-graded modules $I$ and $(v_1,v_2)$ are isomorphic. Thus the equality  $\fdepth I=\depth I$ follows from  Proposition~\ref{complete}. The last equality is a consequence of Proposition~\ref{lower}.
\end{proof}

Next we will show that ideals of Borel type satisfy Stanley's conjecture. For the proof we shall need

\begin{Lemma}
\label{extension}
Let $J\subset I$ be monomial ideals of $S$, and let $T=S[x_{n+1}]$ be the polynomial ring over $S$ in the variable $x_{n+1}$. Then
\[
\depth IT/JT= \depth I/J+1, \quad \fdepth IT/JT= \fdepth I/J+1, \quad \sdepth IT/JT= \sdepth I/J+1.
\]
\end{Lemma}

\begin{proof}
The statement about the depth is obvious since $x_{n+1}$ is regular on $IT/JT$. In order to prove the other two equations we consider the characteristic poset $P_{I/J}$ of $I/J$ as well as the characteristic poset $P_{IT/JT}$ of $IT/JT$. The map $P_{I/J}\to P_{IT/JT}$, $c\mapsto c^*=(c(1),\ldots, c(n),0)$ is an isomorphism of posets with the additional property that $\rho(c)=\rho(c^*)-1$. In particular, if $\mathcal{P}\:\; P_{I/J}=\Union_{i=1}^r[c_i,d_i]$ is a partition of $P_{I/J}$, then $\mathcal{P}^*\:\; P_{IT/JT}=\Union_{i=1}^r[c_i^*,d_i^*]$ is a partition of $P_{IT/JT}$, and the assignment $\mathcal{P}\mapsto \mathcal{P}^*$ establishes a bijection between partitions of $P_{I/J}$ and  $P_{IT/JT}$. Since $\rho(d_i)=\rho(d_i^*)-1$ we see that $\sdepth \mathcal{D(P)}=\sdepth \mathcal{D(P^*)}-1$ for all partitions $\mathcal P$ of $P_{I/J}$.  Therefore the desired equations follow from Corollary~\ref{finite} and Corollary~\ref{ffinite}.
\end{proof}

We would like to remark that Rauf \cite{Ra} proved a similar result for $S/I$. 

\medskip
A monomial ideal is called of {\em Borel type} if it satisfies one of the following equivalent conditions:

\medskip
(i) For each monomial  $u\in I$ and all integers $i,j, s$ with $1\leq j<i\leq n$  and  $s>0$ such that   $x_i^s|u$  there exists an
integer $t\geq 0$ such that $x_j{^t}(u/x_i^s)\in I$.

\medskip
(ii)  If $P\in \Ass(S/I)$, then   $P=(x_1,\ldots, x_j)$ for some  $j$.

\medskip
This class of ideals includes all Borel-fixed ideals (see \cite{Ei}) as well as the squarefree strongly stable ideals \cite{ArHeHi}. Some authors call these ideals also  ideals of nested type \cite{BeGi}. In \cite{Ap} Apel proved that Borel-fixed ideals satisfy Stanley's conjecture. However we could not follow all the steps of his proof. The next result generalizes his statement. For the proof we shall need the following notation: for a monomial $u$ we set $m(u)=\max\{i\: \text{$x_i$ divides $u$}\}$, and for a monomial ideal $I\neq 0$ we set $m(I)=\max\{m(u):u\in G(I)\}$.

\begin{Proposition}
\label{borel} Let $I\subset S$ be an ideal of Borel type. Then $\sdepth S/I\geq \depth S/I$ and $\sdepth I\geq \depth I$. In particular, Stanley's conjecture holds for $I$ and $S/I$.
\end{Proposition}

\begin{proof}
It is shown in  \cite[Proposition 5.2]{HePo} that $S/I$ is pretty clean. This implies  $\sdepth S/I\geq \depth S/I$. In order to prove the second inequality, we use the fact that $S/I$ is sequentially Cohen-Macaulay as was shown in \cite[Corollary 2.5.]{HePoVl}.
Indeed there exists   a chain of ideals $I=I_0\subset I_1\subset\cdots\subset I_r=S$ with the properties that $I_j/I_{j-1}$ is Cohen-Macaulay and $\dim(I_j/I_{j-1})<\dim(I_{j+1}/I_j)$ for all $j$. This chain  of ideals is constructed recursively as follows: let $I_0=I$ and $n_0=m(I_0)$. Suppose that $I_l$ is already defined. If $I_l=S$, then the chain ends. Otherwise, let $n_l=m(I_l)$ and set $I_{l+1}=I_l:x_{n_l}^\infty$. We notice that $n\geq n_0>n_1>\cdots >n_r=0$. It is shown in \cite[Corollary 2.6]{HePoVl} that $\Ext^i_S(S/I,S)\neq 0$ if and only if $i\in\{n_0,n_1,\ldots, n_{r-1}\}$.

Observing that $\depth S/I=\min\{i: \Ext_S^{n-i}(S/I,S)\neq 0\}$, it follows that $\depth S/I=n-n_0$. Therefore $\depth I=n-n_0+1$. Since $G(I)\subset K[x_1,\ldots, x_{n_0}]$, we obtain by applying Lemma \ref{extension} and Proposition \ref{lower} that $\sdepth I\geq n-n_0+1$. Hence we have $\sdepth I\geq \depth I$, as desired.
\end{proof}

Next we compute the $\sdepth$ of an ideal in a special case.

\begin{Proposition}
\label{mci}
Let $I\subset S$ be a monomial complete intersection ideal minimally generated by $3$ elements. Then $\sdepth I=n-1$.
\end{Proposition}

\begin{proof}
Since $I$ is not principal we have $\sdepth I\leq n-1$. In order to prove the statement it is enough, via Corollary~\ref{finite}, to find a partition $\mathcal P$ of $P_I$ such that $\sdepth \mathcal{D(P)}=n-1$. Let $G(I)=\{x^b,x^c,x^d\}$. Since $I$ is a monomial complete intersection we may assume, after a suitable renumbering  of the variables, that $b=(a_1,\ldots,a_i,0,\ldots,0)$, $c=(0,\ldots,0,a_{i+1},\ldots,a_{i+j},0,\ldots,0)$ and $d=(0,\ldots,0,a_{i+j+1},\ldots,a_n)$ with $1\leq i,j,n-i-j$.  We may also assume that $a_k\geq 1$ for all $k=1,\ldots,n$. Indeed, if one of the $a_k$ is zero, then we may use Lemma~\ref{extension} and the proof follows immediately by induction on $n$.

Let $a=b\vee c\vee d=(a_1,\ldots,a_n)$. We claim that
$
\mathcal{P}:\; P_I =B\union C\union D\cup [a,a]
$
is a partition of $P_I$, where
\[
B=\Union_{k=1}^j [b+\sum_{l=1}^{k-1}a_{i+l}\varepsilon_{i+l},a-\varepsilon_{i+k}],
\]
\[
C= \Union_{k=1}^{n-i-j} [c+\sum_{l=1}^{k-1}a_{i+j+l}\varepsilon_{i+j+l},a-\varepsilon_{i+j+k}]
\]
and
\[
D=\Union_{k=1}^i [d+\sum_{l=1}^{k-1}a_{l}\varepsilon_{l},a-\varepsilon_{k}].
\]
It follows then, using Corollary~\ref{finite}, that $\sdepth \mathcal{D(P)}=n-1$, as desired.

In order to prove our claim we first show that the  intervals in  $\mathcal P$ cover $P_I$. In fact, let $e\in P_I$. If $e=a$, then $e\in [a,a]$. Otherwise $e\neq a$ and we may assume that $e\geq b$.
Then $e=(a_1,\ldots,a_i,x_{i+1},\ldots,x_n)$ with $x_k\leq a_k$ for all $k$. Since $e\leq a$ and $e\neq a$ there exists a  $k_0\in\{i+1,\ldots,n\}$ such that $x_{k_0}<a_{k_0}$, and  $k_0$ is minimal with this property. If $k_0\in\{i+1,\ldots,i+j\}$ then $e\in [b+\sum_{l=1}^{k_0-1}a_{i+l}\varepsilon_{i+l},a-\varepsilon_{i+k_0}]\subset B$. Otherwise $e\in C$ by similar arguments.

It remains to be shown that the intervals in $\mathcal P$ are pairwise disjoint. For this we show: (i) the intervals  in each of   $B$, $C$ and $D$ are pairwise disjoint, and (ii)  $B\cap C=B\cap D=C\cap D=\emptyset$.

For the proof of (i) consider for example  the set $B$  (the arguments for $C$ and $D$ are the same). If $j=1$ then we are done.  Otherwise choose two arbitrary intervals in $B$, say $[b+\sum_{l=1}^{k-1}a_{i+l}\varepsilon_{i+l},a-\varepsilon_{i+k}]$ and $[b+\sum_{l=1}^{p-1}a_{i+l}\varepsilon_{i+l},a-\varepsilon_{i+p}]$ with $1\leq k<p\leq j$. Since the $(i+k)$-th component of any vector of the first interval is $<a_{i+k}$ and the $(i+k)$-th component of any vector in the second interval is $a_{i+k}$, it follows that $[b+\sum_{l=1}^{k-1}a_{i+l}\varepsilon_{i+l},a-\varepsilon_{i+k}]\cap [b+\sum_{l=1}^{p-1}a_{i+l}\varepsilon_{i+l},a-\varepsilon_{i+p}]=\emptyset$.

It remains to prove (ii).  Let $e=(e_1,\ldots,e_n)\in B\cap C$. Since $e\in C$, we have $e_k=a_k$ for all $k$ with $k\in\{i+1,\ldots,i+j\}$. On the other hand $e\in B$ implies that there exists $k\in\{i+1,\ldots,i+j\}$ such that $e_{k}<a_{k}$, a contradiction. Hence $B\cap C=\emptyset$. A similar argument can be used to show $B\sect D=C\sect D=\emptyset$.
\end{proof}

We close our paper by stating a conjecture on partitions which follows from  a conjecture of Soleyman Jahan \cite{Ja1}.

We denote  by $\reg M$ the regularity of the graded $S$-module $M$.

\begin{Conjecture}
\label{jahan}
Let $J\subset I$ be monomial ideals. Then there exists a partition $\mathcal{P}\: P_{I/J}=\Union_{i=1}^r[c_i,d_i]$ of the characteristic poset $P_{I/J}$ such that $|c_i|\leq \reg(I/J)$ for all $i$.
\end{Conjecture}

The original conjecture of Soleyman Jahan says that for $I/J$ there exists a Stanley decomposition $\mathcal{D}\: I/J= \Dirsum_{i=1}^rx^{c_i}K[Z_i]$ such that $|c_i|=\deg x^{c_i}\leq \reg(I/J)$ for all $i$. Let $\mathcal{P}$ be the partition of $P_{I/J}$ constructed in Theorem~\ref{sdepth} with the property that $\sdepth \mathcal{D(P)}\geq \sdepth \mathcal{D}$. It follows from the construction of $\mathcal P$  that for each  interval $[c,d]$  of this partition we have $c\in\{c_1,\ldots, c_r\}$. This shows that Soleyman Jahan's conjecture implies Conjecture~\ref{jahan}.

\end{document}